\documentclass[reqno, 12pt]{amsart}
\usepackage{amsmath}
\usepackage[all]{xy}
\usepackage{amsfonts}
\usepackage{a4wide}
\usepackage{amssymb}
\usepackage{amsthm}

\usepackage{epsfig}
\usepackage{color}

\theoremstyle{plain}
\begingroup
\newtheorem{theorem}{Theorem}[section]

\newtheorem{lemma}[theorem]{Lemma}
\newtheorem{proposition}[theorem]{Proposition}

\newtheorem{definition}[theorem]{Definition}
\newtheorem{rmk}[theorem]{Remark}
\newtheorem{example}[theorem]{Example}
\newtheorem{openprob}[theorem]{Open Problem}

\newtheorem{question}[theorem]{Question}
\endgroup

\theoremstyle{remark}
\begingroup
\endgroup

\mathchardef\emptyset="001F

\numberwithin{equation}{section}

\newcommand{\op}[1]{{\rm{#1}}}

\title[Notes on a slice distance for singular $L^p$-bundles]{Notes on a slice distance for singular $L^p$-bundles}

\author[Mircea Petrache]{Mircea Petrache}

\begin{document}

\begin{abstract}
A slice distance for the class of weak abelian $L^p$-bundles in $3$ dimensions was introduced in \cite{PR1}, where it was used to prove the closure of such class of bundles for the weak $L^p$-convergence. We further investigate this distance here, and we prove more properties of it, for example we show that it is H\"older-continuous on the slices. Using the same distance, we give here a notion of a boundary trace, giving a suitable setting for minimization problems on weak bundles. We then state some conjectures and some open questions.
\end{abstract}
 \maketitle
\tableofcontents
\section{Introduction}
In the last decades, an increasing interest has arisen towards analytic tools able to create singular bundles for interesting energies, starting with Donaldson's breakthrough linking $4$-dimensional topology to the study of $SU(2)$-bundles which are critical points of the $L^2$-norm of the curvature studied by Uhlenbeck \cite{Uhl1}. This study is especially interesting in the case of supercritical dimensions, where it gave rise to new questions, among which is the one regarding the possibility of constructing new nontrivial bundles satisfying the Yang-Mills equations by minimizing the natural energy under some topological constraint; see \cite{tian}, \cite{DoTh} and the references therein. For nonabelian bundles in supercritical dimensions, a setting in which the direct method of the Calculus of variations can be applied to obtain critical bundles with singularities is not yet well established (see however the definitions in \cite{isobe2}, \cite{rivkess}). 

\subsection{Defining a Plateau problem for $U(1)$-bundles}
Together with Tristan Rivi\`ere, we started \cite{PR1} an approach which should give the \emph{suitable setting for posing a minimization problem for weak bundles in supercritical dimensions}, in the abelian case of complex $U(1)$-line bundles. Since the Yang-Mills energy in the nonabelian case is the $L^2$-norm of the curvature, we considered the energies defined by $L^p$-norms also in the abelian case. Under such constraint a suitable setting should consist of the following ingredients:
\begin{itemize}
 \item \textit{A class of weak bundles which is closed by sequential weak-$L^p$ convergence of the curvatures:} since as we said the natural energy is the $L^p$-norm of the curvature, the topology giving precompactness of sublevelsets is the weak $L^p$-topology. In particular any minimizing sequence will have a weakly convergent subsequence, and a suitable class of bundles should be closed under this topology.
\item \textit{A suitable notion of boundary value:} if $F$ is the curvature of our bundle, we want to be able to state the minimization problem which could be formally written as follows
\begin{equation}\label{minpb}
\inf\left\{\int_\Omega |F|^pdx: F|_{\partial \Omega}=\phi\right\}
\end{equation}
in a meaningful way. In particular, we would like the weak convergence in the previous point not to disrupt our boundary condition, and to reduce to the usual boundary restriction for locally smooth bundles.
\end{itemize}
As we briefly describe in Sections \ref{closureintro} and \ref{closuresec}, the first point above was solved by the main theorem of \cite{PR1}. The solution of the second point is one of the main results of the present work (See Section \ref{defbdry})
\begin{rmk}
Another possible approach for the creation of nontrivial bundles which are critical for our energy is by minimizing a \emph{relaxed energy} instead, as suggested in \cite{rivkess} and \cite{isoberel}, and in analogy with the case of harmonic maps  \cite{BBC}. In our case a good candidate for such energy would for example be given by
$$
E(F)=\int_\Omega |F|^pdx^3+\sup_{||d\xi||_{L^\infty}\leq 1}\int_\Omega F\wedge d\xi.
$$
\end{rmk}

\subsection{Weak $U(1)$-bundles and vectorfields with integer fluxes}\label{closureintro}
Consider the aim of making a minimization problem as in \eqref{minpb} rigorous for supercritical $U(1)$-bundles. The natural topological invariant of $U(1)$-bundles is the first Chern class, $c_1$, which, for an $U(1)$-bundle $P$ over a compact surface $\Sigma$ is expressible via Chern-Weil theory as
$$
c_1(P)=\int_\Sigma F\in2\pi\mathbb Z\equiv H^2(\Sigma,\mathbb Z),
$$
where $F$ is a curvature on $P$ (see\cite{Zhang}). By identifying the Lie algebra $u(1)$ with $\mathbb R$, we can identify $F$ with a $\mathbb R$-valued $2$-form on $\Sigma$. Then, in the ``supercritical'' dimension $3$, a $2$-form can be interpreted as a curvature if it gives integer volume to almost every closed surface (that integer corresponds to the $c_1$ of a line bundle restricted to the surface). In \cite{rivkess} the following class was first defined:
\begin{definition}\label{defnFZabel}
 We call an $L^p$-integrable $2$-form $F$ defined on a $3$-dimensional domain $\Omega$ a \textbf{curvature of a weak line bundle with group $U(1)$}, if for all $x\in\Omega$ and for almost all $r>0$ such that $B(x,r)\subset\Omega$, there holds
$$
\int_{\partial B(x,r)}i^*F\in\mathbb Z,
$$
where $i:\partial B(x,r)\to\mathbb R^3$ is the inclusion map. We call $\mathcal F_{\mathbb Z}^p(\Omega)$ the class of such $F$.
\end{definition}
\begin{rmk}[vectorfields with integer fluxes]\label{vfif}
 We could associate to a $2$-form $F$ the vectorfield $X$ satisfying
$$
F_p(U,V)=X_p\cdot (U\times V)\quad\text{for all }U,V\in\mathbb R^3,
$$
so that $i_{\partial\Omega}^*F=X\cdot\nu_{\Omega}\op{Vol}_{\partial\Omega}$, $\nu_{\Omega}$ being the outer normal to $\Omega$ and $\op{Vol}_{\partial\Omega}$ being the oriented unit volume $2$-form on $\partial\Omega$. Via this correspondence, we can identify curvatures of $L^p$ weak $U(1)$-bundles as in Definition \ref{defnFZabel} with \textit{$L^p$-vectorfields having integer fluxes} through almost all spheres.
\end{rmk}
Note that in Definition \ref{defnFZabel} no assumption is made a priori, regarding the existence of an underlying topological bundle structure, and we only concentrate on the datum present in our target minimization problem \eqref{minpb}, namely the curvature form $F$, while the presence of an underlying bundle is witnessed just by the Chern class requirement. This is the natural setting where to construct new bundles by minimizing the energy, because the supercritical case is precisely characterized by the possibility of creation of topological singularities, and imposing an initial smooth structure precludes this possibility.
\begin{rmk}
 The sets along which we slice, in Definition \ref{defnFZabel}, are just spheres. By the density result of \cite{rivkess} however, it follows from such definition that automatically such integrality condition is valid on all codimension $1$ closed generic surfaces. In a similar way, the same result for generic surfaces can be achieved starting from a definition which involves slicing sets different from spheres, e.g. cubes, or some other family of sets becoming arbitrarily fine at each point, and allowing a similar density result.
\end{rmk}

\subsection{The closure theorem and the distance on slices: a parallel to the case of currents}\label{closuresec}
The above Definition \ref{defnFZabel} gives a description of bundles in terms of their slices on spheres, and therefore a suggestive parallel can be made with the theory of \textit{scans} present in \cite{HR},\cite{HR1}, \cite{hardtpauw}. A first fruit of this parallel is the idea leading to the proof of the closure theorem in \cite{PR1}:
\begin{theorem}[\cite{PR1}, Main Theorem]\label{closureweakbdl}
Suppose $F_n\in\mathcal F_{\mathbb Z}^p(\Omega)$ as in Definition \ref{defnFZabel} are weakly convergent to some $2$-form $F$. Then $F\in\mathcal F_{\mathbb Z}^p(\Omega)$.
\end{theorem}

The main achievement of the present paper is the definition of the boundary trace in Section \ref{defbdry}. On one hand such definition gives a nontrivial trace, while on the other hand it is preserved under weak convergence. This is due to the intervention of the slice distance directly in its definition, together with the properties described in Section \ref{regandconv}. An interesting consequence of that definition is the following result on the existence of minimizers for the problem \eqref{minpb}:

\begin{theorem}\label{minimexist}
Consider $\Omega=B^3$. Fix an $L^p$-integrable $2$-form $\phi$ on $\partial B^3$ having integer degree. Then the infimum in problem \eqref{minpb} with is achieved, if we interpret the boundary condition $F|_{\partial B^3}=\phi$ as $F\in\mathcal F_{\mathbb Z,\phi}^p(B^3)$, with notations as in Section \ref{defbdry}.
\end{theorem}
\begin{proof}
Fix a minimizing sequence $F_i$ in the class $\mathcal F_{\mathbb Z,\phi}^p(B^3)$. Up to extracting a subsequence we may suppose that $F_i\stackrel{L^p}{\rightharpoonup}F$, which by weak semicontinuity of the $L^p$-norm has energy at most equal to the infimum in problem \eqref{minpb}. Theorem \ref{closureweakbdl} gives then the fact that $F\in\mathcal F_{\mathbb Z}(B^3)$, while Lemma \ref{boundarys} of Section \ref{defbdry} gives $F\in\mathcal F_{\mathbb Z,\phi}^p(B^3)$ as wanted.
\end{proof}
\begin{rmk}
The same result and proof hold in the case of domains $\Omega\neq B^3$ such that $\Omega$ is just bilipschitz equivalent to a smooth domain $\tilde\Omega$. In that case we will have to first define the distance analogous to our $d$ on slices along the sets $\partial\tilde\Omega_r$ which foliate in the usual way a tubular neighborhood of $\partial\tilde\Omega$, then using that distance define the boundary trace exactly as in Section \ref{defbdry}. The proof of H\"older dependence on the parameter $\rho$ proceeds as in Section \ref{regandconv}, and this is enough to prove the results in Section \ref{defbdry}. Proceeding as in Section \ref{lipslices} we can then obtain the same result for $\Omega$, simply by using the bilipschitz equivalence.
\end{rmk}

The fact that minimizers in \eqref{minpb} have finitely many singularities in any compact $K\Subset \Omega$ is proved in \cite{P3}. The boundary regularity (and thus the fact that singularities of minimizers are isolated) will be addressed in \cite{P4}.

We now build an analogy between \cite{AK} (see also the more recent development \cite{AG}), \cite{HR} and the case treated here. Consult \cite{Federer} for the notations of the next paragraph.\\

Recall that in \cite{AK} normal currents $T$ on a metric space $E$ were identified by the property that their slices by Lipschitz functions $f\in\op{Lip}(E,\mathbb R^n)$, having values in the space $\mathcal D_0$ of $0$-dimensional currents
$$
S_f:\mathbb R^n\to\mathcal D_0(E),\quad x\mapsto\langle T,f,x\rangle,
$$
were metric bounded variation (MBV) functions, where $\mathcal D_0$ is endowed with the flat metric
$$
d(\mu_1,\mu_2)=\inf\{\mathbb M(S)+\mathbb M(T): \mu_1-\mu_2=S+\partial T, S\in\mathcal D_0(E), T\in\mathcal D_1(E)\}.
$$
For MBV slice functions it was then proved that the union of the atoms of the slices constituted a \emph{rectifiable} set. Such construction of a rectifiable set tailored on a normal current was the main step for the later closure theorems. The proof of the rectifiability was based on the following estimate valid for $\op{MBV}(\mathbb R^n,S)$ functions $u$, where $(S,d)$ is a weakly separable metric space:
\begin{equation}\label{maxest}
d(u(x),u(y))\leq C(MDu(x)+MDu(y))|x-y|, 
\end{equation}
where $MDu$ is a $L^{1,\infty}$-function related to $Du$ (see \cite{AK}, p.42 for a precise definition).\\

In our case, we slice along spheres, as in Definition \ref{defnFZabel}. Consider a form $F\in\mathcal F_{\mathbb Z}^p(\Omega)$, and a ball $B(x,r)\Subset \Omega$. Then if $i_{x,\rho}:S^2\subset\mathbb R^3\to \partial B(x,\rho)$ is the identification by dilation and translation, we can define the slice function as
$$
S_x:]0,r[\to \mathcal Y, \quad \rho\mapsto i_{x,\rho}^*F,
$$
and the natural choice for the space $\mathcal Y$ is indicated again by Definition \ref{defnFZabel}:
$$
\mathcal Y:=\left\{h\:L^p\text{-form on }S^2:\:\int_{S^2}h\in\mathbb Z\right\}.
$$
\begin{rmk}
 We could also define the slice function on the space of all balls $B\Subset\Omega$ by defining $S(B(x,\rho))= i_{x,\rho}^*F$. This function takes almost everywhere values on $\mathcal Y$.
\end{rmk}

The choice of the suitable distance on $\mathcal Y$ is actually justified by the need of a property like \eqref{maxest}. The picture justifying the definition of the distance depends on the density result of \cite{rivkess} (see \cite{Kessel}, Section 6), and is based on the following result:
\begin{proposition}\label{densitykessel}
 Suppose that $F\in\mathcal F_{\mathbb Z}^p(\Omega)$. Then there exists a $1$-dimensional rectifiable current $I$ such that $\mathbb M(I)\leq C||F||_{L^1}$ for a constant $C$ independent of $F$ and $\langle\phi, \partial I\rangle=\int d\phi\wedge F$ for all $\phi\in C^\infty_c(\Omega)$.
\end{proposition}
The situation between two slices centered at $x$ is shown in Figure \ref{fig:slices}.
\begin{figure}[htp]
\centering
\scalebox{0.5}{\input{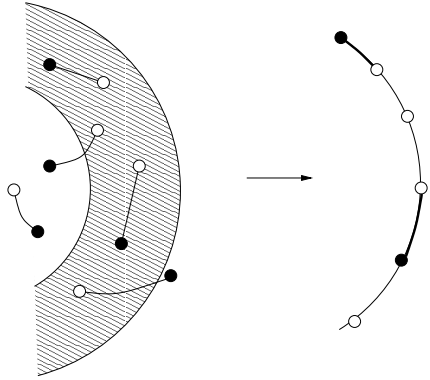}}
\caption{We represent schematically (i.e. we forget for a moment that we are in a $3$-dimensional setting) the use of Proposition \ref{densitykessel}. On the left the portion between two spherical shells is shown, and the current $I$ is represented by a collection of segments, where the boundary components with opposite signs are represented by small balls. An integration in the radial direction reduces us to the picture on the right, where part of the boundary of $I$ projects to a boundary of a current, while for boundaries of components of $I$ which are only partly inside the spherical shell, we obtain a number of Dirac masses. Such number is finite for almost every couple of slices. Applying Stokes' theorem we can compare the slices with the derivative of our $2$-form inside the spherical shell.}\label{fig:slices}
\end{figure}
 Therefore we would expect that the property that $\partial I$ coincides with $d^*F$ in the weak sense be preserved by some projection along the radial segments connecting the slices. Given this picture and this idea, the definition of the following analogous of the flat distance is justified:
\begin{definition}[\cite{PR1}]
Let $h_1,h_2\in\mathcal Y$. We define then the following function
$$
d(h_1, h_2):=\inf\left\{\|\alpha\|_{L^p}:\: h_1-h_2 = d^*\alpha + \partial I + \sum_{i=1}^N n_i\,\delta_{a_i}\right\},
$$
where the infimum is taken over all triples given by a $L^p$-integrable $1$-form $\alpha$, an integer $1$-current $I$ of finite mass, and an $N$-ple of couples $(a_i,n_i)$, where $a_i\in S^2$ and $n_i\in\mathbb Z$.
\end{definition}
In \cite{PR1} it was proved that $d$ as above defines a distance, and satisfies an inequality like \eqref{maxest}, where $MDu$ is replaced by a $L^{p,\infty}$-function depending on $F$ and on the ratio of the two radii defining the slices. The Closure Theorem \ref{closureweakbdl} was proved using the interplay between this metric and the function 
$$
\mathcal N:\mathcal Y\to\mathbb R^+,\quad h\mapsto ||h||_{L^p(S^2)}.
$$
\subsection{Outline of the paper}
In the next section we discuss some definitions of metrics on slices, relying mainly on the density result \cite{rivkess}. In Section \ref{regandconv} we prove that the slice function $S$ is actually H\"older with respect to the above distance, and that it metrizes the weak sequential convergence on slices of bounded norm. A parallel is drawn with the classical case of slices of $W^{1,p}$-vectorfields. In Section \ref{lipslices} we show how to extend the distance to different kinds of slices, where the space $S^2$ in the definition of $\mathcal Y$ is replaced by a Lipschitz domain. This result is used to give a suitable definition of the boundary trace for curvatures belonging to $\mathcal F_{\mathbb Z}^p(\Omega)$ in Section \ref{defbdry}. Then, in the final section, we state several open questions and indicate some further directions of research.

\subsection{Acknowledgements}
I wish to thank my advisor Prof. Tristan Rivi\`ere for introducing me to the topics of this paper and for the many enlightening discussions we had on the topics of this paper. I would also like to thank Prof. Bernd Kirchheim for pointing out the proof of Lemma \ref{kirchlemma}.

\section{Some other definitions of slice distances}\label{eqdef}
We will compare here the distance on $\mathcal Y:=\left\{h\in L^p(S^2):\: \int_{S^2} hd\mathcal H^2 \in\mathbb Z\right\}$ defined by
$$
d(h_1,h_2):=\inf\left\{||\alpha||_{L^p}:\:h_2-h_1=d^*\alpha +\partial I+\sum_{i=1}^Nn_i\delta_{a_i}\right\}
$$
as in the Introduction, to the following function:
$$
d_1(h_1,h_2):=\inf\left\{||\alpha||_{L^p}:\:h_2-h_1=d^*\alpha +\sum_{i=1}^Nn_i\delta_{a_i}\right\}
$$
We define also the following function, in the struggle to free our distance $d$ from the presence of an unknown sum of Dirac masses:
$$
d_2(h_1,h_2)=\lim_{\epsilon\to 0}\inf_{\alpha,A}\left\{||\alpha||_{L^p}:\:\op{spt}\left[(h_2-h_1)-d^*\alpha\right]\subset A, \:A\text{ open },|A|\leq \epsilon\right\}.
$$

The motivations for introducing these objects are as follows:
\begin{enumerate}
 \item Sometimes in applications, as for example in Section \ref{lipslices}, it is easier to deal with definitions in terms of finite, rather than infinite, sets of ``topological'' singularities. This justifies the introduction of $d_1$.
 \item It is natural to ask whether or not our distance $d$ is induced by a norm on $L^p(S^2)$. Candidates for such norms would then be norms which ``don't see small sets'', in particular they should not be sensible to the presence of the singular measures defining $d$. More importantly, having an underlying norm could perhaps help to define new and more natural notions of critical points for our energy. Investigating the relationship between $d$ and $d_2$ seems a reasonable first step in that direction of research. 
\end{enumerate}

We will use the following density result:
\begin{proposition}[see Prop. 1.3 of \cite{P1}]\label{density}
 Fix an exponent $p>1$ and consider the space $\mathcal V_{\mathbb Z}$ consisting of all $L^p$-integrable $1$-forms $\alpha$ on $S^2$ such that 
$$
\int_{S^2}\alpha\wedge d\phi=\langle \partial I,\phi\rangle,\quad\forall\phi\in C^\infty(S^2),
$$
where $I$ is an integer rectifiable $1$-current of finite mass on $S^2$. Then its subspace $\mathcal V_R$ given by the $1$-forms which are smooth outside a discrete (thus finite) set, is dense in the $L^p$-norm.
\end{proposition}
\subsection{The distance $d_1$}
\begin{proposition}\label{eqdistances}
 On $\mathcal Y$ we have $d=d_1$.
 \end{proposition}
\begin{proof}
We clearly have $d_1\geq d$ since the infimum defining $d_1$ is taken on a smaller class. To prove the opposite inequality, fix $h=h_2-h_1$ and consider a minimizing sequence $\alpha_\epsilon$ as in the minimization defining $d$. We then have 
$$
(d^*\alpha_\epsilon) +\left(h-\delta_0\int_{S^2}h\right)=\partial I_\epsilon,\quad ||\alpha_\epsilon||_{L^p}\to d(h_1,h_2).
$$
We consider the function $g$ satisfying the following equation:
$$
d^*dg=h-\delta_0\int_{S^2}h,\quad \int_{S^2}g=0.
$$
By standard elliptic theory we have that $||dg||_{L^p}\leq C||h||_{L^p}$. It follows that 
$$
d^*(\alpha_\epsilon +dg)=\partial I_\epsilon,
$$
and Proposition \ref{density} applies then to $\alpha_\epsilon +dg$, giving us a decomposition 
$$
\alpha_\epsilon +dg=f_\epsilon^k +e_\epsilon^k,
$$
where $f_\epsilon^k\in \mathcal V_R$ is smooth outside of a finite set and $e_\epsilon^k\stackrel{(k\to\infty)}{\longrightarrow}0$ in $L^p$-norm. In particular there exists a measure $\Sigma_\epsilon^k$ of the form $\sum_{i=1}^Nn_i\delta_{a_i}$ as in the definition of $d_1$, for which
$$
d^*f_\epsilon^k = -\Sigma_\epsilon^k=d^*(\alpha_\epsilon+e_\epsilon^k) +h-\delta_0\int_{S^2}h.
$$
Therefore 
$$
h=d^*(\alpha_\epsilon +e_\epsilon^k)+\Sigma_\epsilon^k -\delta_0\int_{S^2}h.
$$
Thus $\alpha_\epsilon+e_\epsilon^k$ are competitors in the infimum defining $d_1(h_1,h_2)$, and as $k\to\infty,\epsilon\to 0$, their $L^p$-norms converge to $d(h_1,h_2)$. This concludes the proof of $d=d_1$.\\
\end{proof}

\subsection{The distances $d_2$ and $d_3$} A possible choice for the set $A$ in the definition of $d_2$ (if we interpret $A$ as the set on which $d^*\alpha$ ``avoids'' as much $L^p$-norm of $h_2-h_1$ as possible) could be some neighborhood of a superlevelset of $|h_2-h_1|$, which gives us a third distance $d_3$:
$$
d_2(h_1,h_2)\leq \lim_{k\to\infty}\inf\left\{||\alpha||_{L^p}:\:h_2-h_1=d^*\alpha\text{ whenever }|h_2-h_1|\leq k\right\}:=d_3(h_1,h_2).
$$
\begin{lemma}
 $d_2$ is a distance and for $h_1,h_2\in\mathcal Y$ there holds $d(h_1,h_2)\geq d_2(h_1,h_2)$.
\end{lemma}
\begin{proof} 
The inequality $d(h_1,h_2)\geq d_2(h_1,h_2)$ follows easily once we know from Proposition \ref{eqdistances} that $d_1=d$, since we can take as the set $A$ a small neighborhood of the singularities in the definition of $d_1$. In particular, it follows that $d_2(h_1,h_2)=0\Leftarrow h_1=h_2$. Being the triangular inequality and the symmetry evident for $d_2$, and since $d_2(h_1,h_2)=0\Rightarrow h_1=h_2$ follows directly from the Lebesgue continuity property of $L^p$-forms, we deduce that $d_2$ is indeed a distance.
\end{proof}

The other inequalities are still to be investigated:
\begin{openprob}
 Is it true that $d=d_2=d_3$?
\end{openprob}

\begin{rmk}\label{rembcs}
We mention here an interesting analogy. A simpler distance similar to $d_2$ was studied in \cite{BCS}, where for probability measures $\mu_1,\mu_2$ on $\Omega\subset\mathbb R^n$ bounded open with smooth boundary the following distance was defined:
$$
D_{\mathcal H}(\mu_1,\mu_2)=\inf_{\sigma\in L^p(\Omega,\mathbb R^n)}\left\{\int_\Omega \mathcal H(\sigma(x))dx:d^*\sigma=\mu_1-\mu_2,\sigma\cdot\nu=0\text{ on }\partial\Omega\right\},
$$
for a class of functions $\mathcal H$ including the case $\mathcal H(x)=|x|^p, p>1$. The connection between our distances and the class of distances $D_{\mathcal H}$ would give an interesting connection to the theory of Optimal Transportation, which would strongly echo with the use of basic Optimal Transportation for ``minimal connections'' connecting singularities of harmonic maps in \cite{Breziscoronlieb}.
\end{rmk}

\section{Regularity and relation to the sequential weak convergence of slices}\label{regandconv}
\subsection{The distance metrizes weak convergence on bounded sequences}
\begin{lemma}
 If for functions $h_n,h_*,h_\infty\in \mathcal Y$ we have that $||h_n||_{L^p}$ is bounded, $d(h_*,h_n)\to 0$ and $h_n\rightharpoonup h_\infty$ weakly in $L^p$, then $h_*=h_\infty$.
\end{lemma}
\begin{proof}
 We observe that from the weak convergence it follows $h_n\stackrel{L^q}{\to}h_\infty$ for all $q<p$. We define then the potentials $\tilde \psi_n, \psi_\infty$ by 
$$
\left\{\begin{array}{l} h_n = \Delta \tilde\psi_n,\int\psi_n=0\\h_\infty=\Delta\psi_\infty,\int\psi_\infty=0\end{array}\right.
$$
and we observe that $\bar\psi_n=\tilde\psi_n-\psi_\infty$ satisfies $\Delta\bar\psi_n\stackrel{L^q}{\to}0$ and $||d\bar\psi_n||_{W^{1,q}}\to 0$.\\

Now take $1$-forms $\alpha_n$ such that
$$
h_n-h_*=d^*\alpha_n +\Sigma_n,\quad ||\alpha_n||_p\to 0
$$
where $\Sigma_n$ is a sum of Dirac masses with integer coefficients, and let $\Delta\phi_n=d^*\alpha_n$. Call also 
$$
\Delta\psi_n=h_*-h_n=h_*-h_\infty +\Delta\bar\psi_n:=\Delta\psi +\Delta\bar\psi_n,
$$
and observe that $||d\psi||_{W^{1,q}}$ is bounded and $||d\phi_n||_p\leq C||\alpha_n||_p\to 0$. Then
$$
\Delta(\psi_n - \phi_n)=h_n -h_* - d^*\alpha_n =\Sigma_n,
$$
and denoting $v_n=d(\psi_n-\phi_n)$ we obtain
$$
\left\{\begin{array}{l}d^*v_n = \Sigma_n\\||v_n||_q\leq C(||\alpha_n||_q+||d\psi||_q+||d\bar\psi_n||_q)\leq C\end{array}\right.
$$
therefore (see the main theorem of \cite{P1}) we can find $u_n\in W^{1,q}(S^2,S^1)$ such that $u_n^*\theta=v_n$, where $\theta$ is the normalized volume $1$-form of $S^1$. The end of the proof goes as in \cite{PR1}, Proposition 3.4, where at the level of the $u_n$ it is possible to find a converging subsequence, and by Sard theorem it is concluded that for a rectifiable $1$-current of finite mass $I_0$ there holds $\partial I_0=h_n-h_*$, thus $h_n=h_*$. See \cite{PR1} for the details.
\end{proof}

\begin{proposition}\label{metrizes}
 If $h_n\in \mathcal Y$ are equibounded in $L^p$, then 
$$
h_n\stackrel{d}{\to}h_*\Leftrightarrow h_n\stackrel{w-L^p}{\rightharpoonup}h_*.
$$
\end{proposition}
\begin{proof}
 Using the fact that a sequence has a limit $h_*$ if and only if each subsequence has a subsequence converging to $h_*$ and the previous lemma, we obtain immediately the ``$\Rightarrow$'' implication.\\
Suppose now $ h_n\stackrel{w-L^p}{\rightharpoonup}h_*$. Then take the potential such that $\Delta\psi_n=h_n-h_\infty$. By the elliptic estimates and the Rellich-Kondrachov theorem, after extracting a subsequence, $d\psi_{\bar n}\to0$ in $L^{p^*}$. The limit is zero independent of the subsequence, so $\alpha_n=d\psi_n$ satisfies 
$$
\left\{\begin{array}{l}h_n-h_\infty=d^*\alpha_n,\\ ||\alpha_n||_p\to0\end{array}\right.
$$
which implies $h_n\stackrel{d}{\to}h_*$.
\end{proof}

\subsection{Regularity on slices}
Here we consider a form $ F\in\mathcal F_{\mathbb Z}^p(\Omega)$ and we want
to compare its slices along $\partial B(x,r),\partial B(x',r')\subset\Omega$.
This will be possible only under some condition on $|x-x'|,r-r'$m but we
formulate the condition later.\\

The slices will be given by a function (defined a.e.) $h:\Omega\times\mathbb
R^+\to \mathcal Y\subset L^p(S^2)$, where $h(x,r)$ is the function on $S^2$
corresponding to the restriction of $ F$ to $\partial B(x,r)$, after
a homothety and an identification $\wedge^2S^2\simeq \wedge^0S^2$.\\

We consider the following function $A:S^2\times[0,1]\to \Omega$: 
$$
A(\sigma,t)=t(x-x')+x' +\left[t(r-r') +r'\right]\sigma:=x_t+r_t\sigma.
$$
Suppose that $A$ is a diffeomorphism onto its image. Then 
$A^* F\in \mathcal F_{\mathbb Z}(S^2\times[0,1])$, and we will build out of
it a competitor for the infimum in the definition of $d(h(x,r),h(x',r'))$. 
Consider
$$\bar F(\sigma)=\frac{1}{|r-r'|}\int_0^1
 F_{x_t,r_t}^\parallel(\sigma)dt,\quad\text{where
} F_{x_t,r_t}(\sigma):=r_t^2 F(x_t+\sigma r_t).
$$
Here $F_{x_t,r_t}^\parallel$ indicates the component of $F_{x_t,r_t}$ parallel to the volume form of the sphere $\partial B_{r_t}(x_t)$. Along the lines of Proposition 4.1 of \cite{PR1} (See also Proposition \ref{densitykessel} and Figure \ref{fig:slices} in the Introduction), we can show that this gives a
competitor. We then introduce the reparameterization $\rho=r_t$ and we compute:
\begin{eqnarray*}
 \int_{S^2} |\bar F|^p(\sigma)d\sigma
&=&
\int_{S^2}\left(\int_{r'}^r F_{x_\rho,\rho}
^\parallel(\sigma)d\rho\right)^pd\sigma\\
&\leq&|r-r'|^{1-\frac1p}\int_{r'}^r\int_{S^2}| F_{x_\rho,\rho}|^pd\sigma
d\rho.
\end{eqnarray*}
In order to compare this with the norm of $ F$, we first find
$$
DA(\sigma,t)=\left([t(r-r')+r']Id_{TS^2}|x-x' +
(r-r')\sigma\right)=\left(r_tId_{TS^2}|x-x' + (r-r')\sigma\right).
$$
Then (assuming $B'\subset B$ for the moment) we pull back the function
$| F|^p$:
\begin{eqnarray*}
\int_{B\setminus B'}| F|^pd\mathcal H^3&=&\int_{A^{-1}(B\setminus
B')}| F|^p\circ A |DA|\\
&=&\int_0^1\int_{S^2}| F(x_t+r_t\sigma)|^pr_t^2|r-r'+\langle\sigma,
x-x'\rangle|d\sigma dt.
\end{eqnarray*}
 We can
now formulate our hypothesis on the slices:
$$
(H)\quad |x-x'|\leq \frac{1}{2}(r-r'), 1\geq r>r'.
$$
Under this hypothesis, (since $|\sigma|=1$) we can estimate
\begin{eqnarray*}
 \int_{B\setminus B'}| F|^pd\mathcal
H^3&\geq&\frac12(r-r')\int_0^1\frac{1}{r_t^{2p-2}}\left(\int_{S^2}| F_{x_t,
r_t }|^pd\sigma\right) dt\\
&=&\frac12 \int_{r'}^r\frac{1}{\rho^{2p-2}}\left(\int_{S^2}| F_{x_\rho,\rho
}|^pd\sigma\right) d\rho\\
&\geq&\frac12 \int_{r'}^r\left(\int_{S^2}| F_{x_\rho,\rho
}|^pd\sigma\right) d\rho\quad\text{ if }\rho\leq 1.
\end{eqnarray*}
Observe that $ F_{x_\rho,\rho}$ is the Poincar\`e dual of $h(x_\rho,\rho)$.
We introduce the function 
$$
F:\Omega\times \mathbb R^+\stackrel{h}{\to}
\mathcal Y\stackrel{||\cdot||^p_{L^p}}{\to}\mathbb R^+,
$$ 
and we can finally  write
\begin{eqnarray*}
d(h(x,r),h(x',r'))
&\leq&|r-r'|^{1-\frac{1}{p}}\left(\int_{r'}^rF(x_\rho,
\rho)\right)^ { 1/p }\\
\text{(Under hypotheses
(H)..)}&\leq& 2|r-r'|^{1-\frac{1}{p}}\left(\int_{B\setminus
B'}| F|^p\right)^{1/p}.
\end{eqnarray*}
Combining the basic estimate above for a couple of segments, we obtain
H\"olderianity.

\begin{proposition}\label{holder}
 The slice-function $h: \mathcal A:=B_{\frac12}\times \mathbb
R^+\cap\{(x,r):\:B(x,r)\subset B_1\}\to L^p(S^2)$ defined above is
H\"older-$(1-1/p)$-continuous on with respect to the distance $d$, and its
H\"older constant is bounded by the $L^p$-norm of $ F$. 
\end{proposition}

\begin{proof}
 We want to see how our estimates worsen if instead of connecting $B=(x,r),B'=(x',r')$ along a segment, we use a polygonal curve. Consider then $\gamma$, consisting in a union of segments $\{S\}$, each ow which satisfies (H). For a given segment $S=[\underline S,\overline S]$ (where $\overline S=(x',r')$ is the end with the largest radius) we call $A_S:=B_{\overline S}\setminus B_{\underline S}$. For a given $S$, call $|S_r|$ the difference of the radii of $\underline S,\overline S$. We then have the following estimate, following the same reasoning as above:
\begin{eqnarray*}
2|| F||_{L^p(A_S)}^p&\geq&\int_{\underline S}^{\overline S}\frac{1}{\rho^{2p-2}}F(s_\rho)d\rho\\
&\geq&\overline S^{2-2p}|S_r|^{p-1}d(h(\overline S),h(\underline S))^p,
\end{eqnarray*}
and summing up and using the triangle inequality, 
$$
2\#\{S\}|| F||_{L^p}\sum_{S\in\gamma}|S_r|^{1-\frac1p}\geq d(h(B),h(B')).
$$

Because of this estimate, the question is how we can join $B,B'$ by some polygonal $\gamma$ which stays in the allowed set $\mathcal A$ and is made of segments verifying (H), such that $\#\{S\}$ is as small as possible and $\sum_{S\in\gamma}|S_r|^{1-\frac1p}$ is bounded above.\\

We will see that $N$ can be bounded by $4$ because we don't need more than $4$ segments, and that $\max_{S\in\gamma}|S_r|$ is bounded  by $2 |B'-B|$ (also in this case it's optimal to have a few long segments rather than many short ones). We just briefly describe the kind of $\gamma$ we use for the estimates.\\

The worst case that we can face is the one where $B,B'$ are on $\partial \mathcal A$, have the same $r$-coordinate, and are as far from each other as possible. If they are on the part where $x,x'\in\partial B_{\frac12}$ with $r<\frac14$ then we can take $\gamma$ to start from $B$ and go up in the $r$-direction with slope $2$ until it touches $\partial\mathcal A$, then down until close to $0$ radius and center $x=0$, then do the same symmetrically, building up an $M$-shaped graph. If $r\geq \frac14$, then it's better to first go down then up, making a symmetric $W$-shaped graph. If instead $x,x'\in\partial\mathcal A\setminus\partial B_{\frac12}$ then again a $W$-shaped graph is the best option, and if $r$ is large enough a $V$-shaped graph will be even better.\\
It is easy (but tedious) to verify that the above constructions verify the estimate on $|S_r|$. We thus end up with the following bound:
$$
16|| F||_{L^p}|B-B'|^{1-\frac1p}\geq d(h(B),h(B'))
$$
\end{proof}
\begin{rmk}\label{distancefig}
 We observe that in general, even though $d$ is H\"older on the slices, Proposition \ref{metrizes} does not apply, to give weak continuity on the slices, because the norm boundedness is not verified. This is already clear in the case where the form $F$ is the radial form $F_x(V,W)=\frac{x}{|x|}\cdot V\times W$. Then consider the slices $S_{1+\rho}$ along $\partial B(1+\rho,(0,0,1)),\rho\in[-\epsilon,+\epsilon]$ (see Figure \ref{slicesunbound}). Since these spheres look almost flat near $(0,0,0)$ for small $\epsilon$ and the integral of $F$ on the portion of a given slice just depends on the solid angle covered by that region, we easily see that the $L^p$-norm of the slice $S_{1+\rho}$ on a small ball near the singularity grows like $\rho^{2-2p}$, i.e. blows up.
\end{rmk}
\begin{figure}[htp]
\centering
\scalebox{0.5}{\input{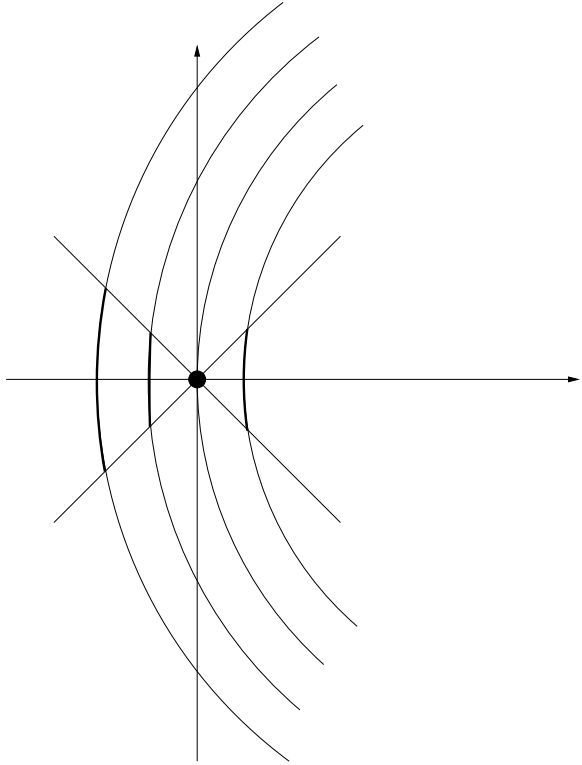}}
\caption{We represent schematically the slices passing near the origin. The areas of the thick regions behave like $\rho^2$ and the integral of $F$ on them is constant and positive, so $|i^*F|\sim \frac{1}{\rho^2}$ and the $L^p$-norms of the slices is thus $\gtrsim \rho^{2-2p}$.}\label{slicesunbound}
 
\end{figure}

\subsection{A simplified proof of the closure theorem}\label{simplif}
Propositions \ref{holder} and \ref{metrizes} allow a simplification of the proof of the Closure Theorem \ref{closureweakbdl}, which avoids using the stronger Theorem 5.1 of \cite{PR1}. We state here the crucial result from which Theorem \ref{closureweakbdl} follows at once, and we give a new proof of it.
\begin{lemma}[Main step of the Closure Theorem]
Let $p\in]1,3/2[$ as above. Suppose that the $2$-forms $F_n\in\mathcal F_{\mathbb Z}^p(\Omega)$ are weakly convergent to a $2$-form $F\in L^p(\Omega)$. Given $B_r(x)\subset\Omega$, consider the slice functions go $F$, $S:[r/2,r]\to L^p(S^2)$, given by $S(\rho):=i^*_{x,\rho}F$. Then for almost all $\rho\in[r/2,r]$, $S(\rho)\in \mathcal Y$, i.e. the integer flux condition is preserved.
\end{lemma}
\begin{proof}
We suppose w.l.o.g. that $x=0$. By lower semicontinuity of the norm, we may suppose that $||F_n||_{L^p(B_r\setminus B_{r/2})}\leq C$. By Proposition \ref{holder}, the slice functions $S_n$ defined as $S$ in the statement of the Lemma, but with $F_n$ instead of $F$, are equi-H\"older with respect to our metric $d$. This means that we can extract a pointwise convergent subsequence.\\
It is evident that the deformation factor of the $L^p$-norm, coming from the fact that $i_{0,\rho}$ are dilations, is bounded. Fubini's and Chebychev's theorems imply then, that we may restrict to a subset of $\rho \in[r,r/2]$ on which the $L^p$-norms of the $S_n(\rho)$ stay bounded. This is just the situation where Proposition \ref{metrizes} applies. Therefore the slices converge weakly almost everywhere, and testing them on the constant function $1$, we see that their integer degrees also converge. Therefore $S(\rho)$ has integer degree on $S^2$, as wanted.
\end{proof}

\section{The case of Lipschitz slices}\label{lipslices}
We consider here the problem of extending the definition of the distance $d$ to the case of slices different from spheres. The main motivations for this extension are the following:
\begin{itemize}
\item A natural question regarding the class $\mathcal F_{\mathbb Z}^p$ is whether or not the condition that the integrality is required on spheres can be replaced by a condition on different kinds of surfaces. A particularly interesting case would be one in which the slicing sets tile space, as is the case for the surfaces of cubes.
\item The definition of the boundary condition in Section \ref{defbdry} will be based on slicing. Therefore having more general slice models will allow defining the trace on more general domains.
\end{itemize}
 Given a bilipschitz map $\Psi:S^2\to\Sigma$, we thus define the following distance between $L^p$-integrable $2$-forms on $\Sigma$:
$$
d_{\Psi}(h_1,h_2)=d_{S^2}(\Psi^* h_1, \Psi^* h_2).
$$
We observe that the pullback by bilipschitz functions preserves the integrability class, since 
$$
|(\Psi^* h)_x|=\sup_{|v|\leq 1,|w|\leq 1}h_{\Psi(x)}(d\Psi_xv,d\Psi_xw)\leq ||d\Psi||_{\infty}^2|h_{\Psi(x)}|,
$$ 
and the same holds with $\Psi^{-1}$ instead of $\Psi$.
Analogous estimates imply that different bilipschitz maps induce equivalent distances:
\begin{proposition}\label{bilipeq}
 Suppose $\Psi_1,\Psi_2:S^2\to\Sigma$ are bilipschitz maps. Then $d_{\Psi_i}$ are distances and they are equivalent:
$$
C^{-1}d_{\Psi_1}\leq d_{\Psi_2}\leq Cd_{\Psi_1}.
$$
Moreover the constant $C$ depends only on the Lipschitz constants of $\Psi,\Psi^{-1}$.
\end{proposition}
\begin{proof}
 The fact that $d_{\Psi_i}$ satisfy the triangular inequality and the reflexivity follow at once from the analogous properties of $d$. The nondegeneracy $d_{\Psi_i}(h_2,h_2)=0\Leftrightarrow h_1=2_h$ is a consequence of the equivalence stated in the Proposition, for the case $\Psi_1=id_{S^2}$, and $\Sigma=S^2$. We also observe that if we prove the Proposition for this special case, the general case will follow by transitivity of the equivalence between distances. Thus we will consider just this case.\\

We will work with the equivalent definition of $d_1$ as in Section \ref{eqdef}.\\

Fix $h_1,h_2\in \mathcal Y$, and consider a competitor $\alpha$ in the definition of $d_1(h_1,h_2)$. In other words we have, in the $2$-form setting, that if $h=h_2-h_1$ and $\Sigma$ represents a finite sum of Dirac masses, then
$$
d(*\alpha)=h+*\Sigma, \text{ i.e. }\forall\phi\in C^\infty(S^2),\int\phi d(*\alpha) = \int \phi h +\langle \Sigma,h\rangle,
$$
where $*$ represents the Hodge star operator with respect with the standard metric.
The crucial observation is that all the objects above extend naturally to the space of Lipschitz functions, and it is equivalent to use just $\phi\in \op{Lip}(S^2)$ instead of $\phi\in C^\infty(S^2)$ above. By replacing $\phi$ by $\phi\circ\Psi\circ\Psi^{-1}$ and changing variable, we obtain (recall that $\Psi_\#\Sigma$ is the image measure):
$$
\int d\left(*(\Psi^*\alpha)\right)\phi\circ\Psi = \int \Psi^*h \phi\circ\Psi + \langle\Psi_\#\Sigma,\phi\circ\Psi\rangle.
$$
Since $\Psi$ is bilipschitz, it is a bijection of $\op{Lip}(S^2)$ into itself, and thus we see that $\Psi^*\alpha$ is a competitor for the distance 
$d_\Psi(h_1,h_2)$.\\

Now observe as above that $|\Psi^*\alpha|_x\leq||d\Psi||_\infty|\alpha|_{\Psi(x)}$, which leads to the conclusion that 
$$
\int_{S^2}|\Psi^*\alpha|^p_xdx\leq||d\Psi||_\infty^p\int|\alpha|_{\Psi(x)}^pdx\leq ||d\Psi||_\infty^p||d\Psi^{-1}||_\infty^2\int|\alpha|^p_ydy.
$$
The same holds also with $\Psi^{-1}$ instead of $\Psi$, so the infimum in the definition of $d$ is comparable with the one in the definition of $d_\Psi$.
\end{proof}

\section{Definition of the boundary value}\label{defbdry}
Let $\Omega\subset\mathbb R^3$ be an open bounded smooth domain. We consider here the class $\mathcal F_{\mathbb Z}^p(\Omega)$ as described in Section \ref{lipslices}. Such class consists of all $L^p$-integrable $2$-forms $F$ such that for generic $2$-cycles $S$ bilipschitz-equivalent to $S^2$, there holds
$$
\int_S F\in\mathbb Z.
$$
In this Section we would like, given a smooth $2$-form $\varphi$ on $\partial \Omega$, to find a suitable class $\mathcal F_{\mathbb Z,\varphi}^p(\Omega)$, which satisfies the following three conditions:
\begin{itemize}
\item\textbf{(closure)}\label{wellp} for any $L^p$-regular $2$-form $\varphi$ on $\partial \Omega$, the class $\mathcal F_{\mathbb Z,\varphi}^p(\Omega)$ is closed by sequential weak $L^p$-convergence.
\item\textbf{(nontriviality)}\label{nontr} if $\varphi\neq\psi$ are two $L^p$-regular $2$-forms on $\partial \Omega$, then $\mathcal F_{\mathbb Z,\varphi}^p(\Omega)\cap\mathcal F_{\mathbb Z,\psi}^p(\Omega)=\emptyset$.
\item\textbf{(compatibility)}\label{compa} for any smooth $2$-form $\varphi$, $\mathcal F_{\mathbb Z,\varphi}^p(\Omega)\cap\mathcal R^\infty$ are exactly the $2$-forms $F\in\mathcal R^\infty$ such that $i^*_{\partial \Omega}F=\varphi$, where $i_{\partial \Omega}$ is the inclusion map.
\end{itemize}
For general $L^p$-forms (i.e. without the restriction of belonging to $\mathcal F_{\mathbb Z}^p$) no such class can exist, even if in the closure requirement above we had required \emph{strong} convergence. Indeed, let $F,G$ be different smooth forms, and consider $f_n:[0,\infty[\to[0,1], f_n=\chi_{[1/n,\infty[}$. Then $F_n(x):=F(x)+f_n(\op{dist}(x,\partial\Omega))(G(x)-F(x))$ satisfy $F_n\stackrel{L^p}{\to}G$. Then by compatibility $F_n$ and $F$ should have the same trace, and so by closure $G$ and $F$ should have the same trace, contradicting nontriviality.\\
At the other extreme, for \emph{locally exact} $L^p$-forms, using the Poincar\'e Lemma, we have
$$
dF=_{loc}0\Longrightarrow F=_{loc}dA,\:\:A\in W^{1,p}_{loc}.
$$
Thus we can impose the boundary condition directly on the restrictions to $\partial\Omega$ of $ W^{1,p}$-regular ``local primitives'' $A$, using classical trace theorems, and we easily obtain all the above properties.\\
Our new space $\mathcal F_{\mathbb Z}^p(\Omega)$ is an intermediate space between the two extrema above, escaping both the above reasonings. We therefore use the approach which is a natural consequence of \cite{PR1}, namely we use the distance $d_S$ between $2$-forms on cycles $S$, as in Section \ref{bilipeq}. Up to applying a bilipschitz deformation, we may assume that we have $\Omega=B^3$, and we will define the boundary condition in this case first.\\

We use the distance $d$ to compare the boundary datum with the slices of our forms belonging to $\mathcal F_{\mathbb Z}^p$. We call $F(x+\rho)$ with variable $x\in S^2$ the form on $S^2$ corresponding to the restriction to $\partial B_{1-\rho}$ of the form $F$. 
We define the right class $\mathcal F_{\mathbb Z,\varphi}(B^3)$ via the continuity requirement
\begin{equation}\label{cond}
d(F(x+\rho'),\varphi(x))\to 0\text{, as }\rho'\to 0^+.
\end{equation}
It is clear that the definition \eqref{cond} satisfies the \textit{nontriviality} and \textit{compatibility} conditions above, since $d(\cdot,\cdot)$ is a distance and since for $\mathcal R^\infty$ having smooth boundary datum implies that in a neighborhood of $\partial B^3$ the slices are smooth and converge in the smooth topology to $\varphi$. The validity of the \emph{well-posedness} is a bit less trivial, therefore we prove it separately.
\begin{lemma}\label{boundarys}
 If $F_n\in\mathcal F_{\mathbb Z, \varphi}(B^3)$ are converging weakly in $L^p$ to a form $F\in\mathcal F_{\mathbb Z}(B^3)$ then also $F$ belongs to $\mathcal F_{\mathbb Z, \varphi}(B^3)$.
\end{lemma}
\begin{proof}
 By weak semicontinuity of the $L^p$ norm we have that $F_n$ are bounded in this norm. We then have (since $F_n\in \mathcal F_{\mathbb Z}$) for all $n$: what we need is precisely $||F_n||_{L^p(B_1\setminus B_{1-h})}\leq C$.\\

Therefore by Proposition \ref{holder} the $F_n$ are $d$-equicontinuous, so a subsequence (which we will not relabel) of the $F_n$ converges to $F_\infty\in\mathcal F_{\mathbb Z}$, i.e. for all $\rho'\in[0,\rho]$ the forms $F_n(x+\rho')$ are a Cauchy sequence for the distance $d$. This is enough to imply that $F=F_\infty$. We observe that $F$ is just defined up to zero measure sets, but it has a $d$-continuous representative. By continuity it is clear that $F$ still satisfies \eqref{cond}.
\end{proof}
The same proof also gives an apparently stronger result:
\begin{proposition}\label{boundarys2}
 If $F_n\in\mathcal F_{\mathbb Z, \varphi_n}(B^3)$ are converging weakly in $L^p$ to a form $F\in\mathcal F_{\mathbb Z}(B^3)$ then the forms $\varphi_n$ converge with respect to the distance $d$ to a form $\varphi$ and also $F$ belongs to $\mathcal F_{\mathbb Z, \varphi}(B^3)$.
\end{proposition}

\begin{rmk}
At a previous stage in the preparation of this paper, before remarking the $d$-h\"olderianity of the slices, we had prepared a different definition of the boundary value, based on the $L^{p,\infty}$-bound of the modulus of lipschitzianity of $d$ given in Theorem 5.1 in \cite{PR1}. Assuming just such bounds, with $p>1$, Proposition \ref{boundarys2} remains true, if we replace condition \eqref{cond} by the following approximate continuity requirement 
$$
\text{for all }\epsilon>0,\lim_{\rho\to 0^+}\frac{|[0,\rho]\cap A_\epsilon|}{\rho}=0,\text{ where }A_\epsilon:=\{\rho':\:d(F(\cdot+\rho),\varphi))>\epsilon\}.
$$
\end{rmk}

We give the result also in the formalism of vectorfields with integer fluxes mentioned in Remark \ref{vfif}:
\begin{proposition}
 Let $L^p_{\mathbb Z}(B^3)$ be the class of vectorfields with integer fluxes described in Remark \ref{vfif}. Let $\hat\rho$ be the radial vectorfield defined outside the origin of $\mathbb R^3$. For $(x,\rho')\in S^2\times]0,1[$, define $\xi(x+\rho'):=\hat\rho\cdot X(x(1-\rho'))$. For a given $L^p$-regular function $\phi$ defined on $\partial B^3$ we define then the class $L^p_{\mathbb Z,\phi}(B^3)$ via the continuity requirement 
$$
d(\xi(x+\rho'),\phi(x))\to 0\text{ as }\rho'\to 0^+,
$$ 
where we identify $2$-forms on $S^2$ to functions via the Hodge-star duality with respect to the standard metric.\\

With this definition we have the following two properties:
\begin{enumerate}
 \item If $X_n\in L_{\mathbb Z, \varphi}^p(B^3)$ are converging weakly in $L^p$ to a form $X\in L_{\mathbb Z}^p(B^3)$ then also $X$ belongs to $L_{\mathbb Z, \varphi}^p(B^3)$.
 \item If $X_n\in L_{\mathbb Z, \varphi_n}^p(B^3)$ are converging weakly in $L^p$ to a form $X\in L_{\mathbb Z}^p(B^3)$ then the forms $*\varphi_n$ (where $*$ is the Hodge star with respect to the standard metric) converge with respect to the distance $d$ to a function $*\varphi$ and $X$ belongs to $L^p_{\mathbb Z, \varphi}(B^3)$.
\end{enumerate}
\end{proposition}
\begin{proof}
 The correspondence described in Remark \ref{vfif} translates the language of forms into that of vectorfields. Under this translation the restriction operation $F\mapsto i^*_{S^2}F$ corresponds to the operation $X\mapsto \nu_{S^2}\cdot X$. The closure of $L^p_{\mathbb Z}(B^3)$ under weak convergence being proved in \cite{PR1}, we are just left to prove the preservation and convergence of the boundary condition. The needed results are proved in the language of differential forms in Lemma \ref{boundarys} and in Proposition \ref{boundarys2} respectively.
\end{proof}

\begin{rmk}
 As already pointed out, the definition of the distance as in Section \ref{lipslices} allows to extend the definition of the boundary value to arbitrary domains 
\end{rmk}

\section{Some questions and conjectures}\label{questconj}
\subsection{First steps towards a compatibility condition for slices}
Since we used slices to define our class of weak curvatures $\mathcal F_{\mathbb Z}(\Omega)$, it is natural to try to go one step further and try to represent a form $F\in\mathcal F_{\mathbb Z}(\Omega)$ by its slices. This kind of problem seems to represent an unexplored area of research, related perhaps to integral geometry. We were not able to find any example of similar problems in the literature. Therefore in the following subsections we will attempt to formalize the main questions which have arisen.
\subsubsection{Slices on rectifiable cycles and genericity}
Consider a $2$-form $ F \in\mathcal F_{\mathbb Z}^p(\Omega)$ which is bounded in $L^p$-norm. Given a Lipschitz $2$-cycle $C=\partial K$ on $\Omega$, chosen in a ``generic'' way such that $i^*_C F $ is in $L^p(C,\mathcal H^2)$ and that (in the duality between $2$-cycles and $2$-forms)
$$
\langle C, F \rangle\in\mathbb Z
$$
we can associate
$$
C\mapsto h(C):=i^*_C F \in Y_C,
$$
where $Y_C$ is the set of $2$-forms $h$ such that
\begin{itemize}
 \item $h$ is $L^p$-integrable w.r.t. the surface measure on $C$,
 \item $h$ is $\mathcal H^2$-a.e. the dual of the unit tangent $2$-vector $\vec C$ to $C$,
 \item $\langle C,h\rangle$ is an integer.
\end{itemize}
We have still to explain what the requirement that $h$ be defined only for ``generic'' cycles should mean. For that purpose, denote by $\mathcal C$ the \emph{fixed set of Lipschitz cycles} on which our compatibility theory will be defined (useful choices may vary from the set of all spheres to the set of all Lipschitz cycles). The domain of definition of $h$ above should then given by $\mathcal C\setminus R_F$ for some set $R_F$, possibly depending on $F$, which belongs to an admissible class of residual sets defined as follows.
\begin{definition}\label{residual}
We will call an \textbf{admissible class of residual sets} a class $\mathcal R\subset\mathcal C$ satisfying the following:
\begin{itemize}
 \item Suppose that $(C_x)_{x\in[-\epsilon,\epsilon]}$ is a Lipschitz foliation by Lipschitz cycles $C_x\subset\Omega$ i.e. there is a Lipschitz cycle $C\in\mathcal C$ and a bilipschitz parameterization $\Psi:C\times[-\epsilon,\epsilon]\to\cup_{x\in[-\epsilon,\epsilon]} C_x$ sending $C\times\{x\}$ to $C_x$. Then there exists $\delta\leq \epsilon$ such that the intersection with $\mathcal C$ of set of cycles $C_x$ corresponding to choices of $x$ inside a subset of of zero Lebesgue measure of $]-\delta,\delta[$, should form a set belonging to $\mathcal R$.
\end{itemize}
  Once we fixed an admissible class of residual sets, we will call the complement of a residual set \textbf{generic}.
\end{definition}

\subsubsection{The compatibility question}
Consider the question of slice compatibility for a class of slicing cycles $\mathcal C$. The starting observation is that even in simple cases, not all applications 
\begin{equation}\label{slices}
k:\mathcal C\to \mathcal Y_{\mathcal C}:=\cup_{C\in\mathcal C}Y_C
\end{equation}
can be represented as slices $h$ of an underlying form $ F \in\mathcal F_{\mathbb Z}^p$:

\begin{lemma}\label{notcomp}
 Assign to each cycle $C=[\partial B(x,r)]$ the form $h(C)\in Y_C$ equal to $\psi_C^*h(\mathbb S^2)$ of a fixed nonzero $2$-form $h(\mathbb S^2)\in Y_{\mathbb S^2}$, where $\psi_C:C\to\mathbb S^2$ is the similitude bijection. The so-obtained function 
$$
h:\mathcal C=\{[\partial B(x,r)]:\:x\in\Omega, r\in]0,\op{dist}(x,\partial\Omega)\}\to\mathcal Y_{\mathcal C}
$$
cannot satisfy $h(C)=i^*_CF$ for generic $C\in\mathcal C$. 
\end{lemma}
\begin{proof}
Assume for a moment that there exists such $2$-form $F\in\mathcal F_{\mathbb Z}^p(\Omega)$. Then for any fixed $M>0$ we would have $|F|\geq M$ almost everywhere on $\Omega$. Indeed fix $\epsilon>0$ such that $|E_\epsilon|>\epsilon$, where $E_\epsilon:=\{|h|>\epsilon\}$. Then consider the sets $S(x,r):=\cup_{\rho\leq r}\psi_{B(x,r)}^{-1}(E_\epsilon)$, with the constraint $r<\sqrt{\epsilon/ M}$. These sets form a fine covering of $\Omega$, and if $h(\mathbb S^2)=(\psi_C^{-1})^*i^*_CF$ for almost all $C$ in the definition of $S(r,x)$, then $|F|$ must be larger than $M$ almost everywhere on $S(x,r)$. By extracting a (not necessarily disjoint) countable cover of $\Omega$ by sets $S(x,r)$ up to zero Lebesgue measure, we obtain that $|F|\geq M$ almost everywhere. By the arbitrariness of $M$ we obtain that $F$ cannot be in $L^p$, thus contradicting our assumption.
\end{proof}

\begin{rmk}\label{star}
Suppose that $\mathcal C$ is a family of cycles such that for almost all $x\in\Omega$ the tangent spaces $(T_x C)_{x\in C\in\mathcal C}$ span the Grassmannian $G(2,1)$ of $2$-planes. Then for any $k$ as in \eqref{slices} there is at most one $2$-form $F$ such that $k=k_F$. Indeed, fixing the restrictions $i^*_CF$ at some point $x$ along three linearly independent tangent planes relative to three choices of $C$, automatically fixes the value of $F$ at $x$.
\end{rmk}

The compatibility requirement between $k$ and $F$ following from Remark \ref{star} depends on the pointwise behavior of the single slices. We would like to find a more geometric condition $(C)$ which can be tested by looking only at the function $k$ as in \eqref{slices}. See Question \ref{star5} for such an example. The wanted condition $(C)$ should also satisfy the following properties.
\begin{definition}\label{compatdef}
 Suppose that $(C)$ is a property of a function $k$ as in \eqref{slices} for a given set of cycles $\mathcal C$. We say that $(C)$ is a \textbf{compatibility condition} if the following are true:
\begin{enumerate}
 \item If $ F \in\mathcal F_{\mathbb Z}^p(\Omega)$ for some $p\in]1,3/2[$ then the function $k_F $ which to a generic $C\in \mathcal C$ associates the slice of $F$ along $C$, satisfies $(C)$.
\item If $ F_i \in\mathcal F_{\mathbb Z}^p(\Omega)$ are a sequence converging $L^p$-weakly to a form $F$, then also $k_F$ satisfies $(C)$.
 \item Whenever $k$ satisfies $(C)$, there exists a $ F \in\mathcal F_{\mathbb Z}^p(\Omega)$ such that $k=k_ F $.
\end{enumerate}
\end{definition}
Since we ``know much more'' about $\mathcal F_{\mathbb Z}^p$ than about weak convergence or about slice functions, in general the first point above should prove relatively easier to check.

\begin{example}\label{tooweak}
In \cite{PR1}, a $(C)$ satisfying the first condition was given, and consisted in asking that for $\mathcal C=\{\partial B(x,r):\:x\in\mathbb R^3, r>0\}$ (where the generic sets are the algebra generated by the ones of the form $\{\partial B(x,r):\:r\in N\}$ s.t. $ \mathcal L^1(N)=0$), the integral of $k(C)$ be an integer. This is just the definition of $\mathcal F_{\mathbb Z}^p$. As shown by Lemma \ref{notcomp}, this candidate for condition $(C)$ is too weak to satisfy the second property above.
\end{example}

\subsubsection{A simple geometric candidate for compatibility} We consider still the case where $\mathcal C$ consists of all spheres contained in $\Omega$. If we imagine that our form $F$ has only finitely many singularities, then the integral $\int_C F$ along each cycle corresponds to an algebraic sum of the degrees associated to the singularities situated in the interior of $C$. Now consider two intersecting spheres, $C',C''$ and suppose that their intersection is a circle $D$. If we assume that none of the singularities of $F$ is on $C'\cup C''$, we will have then that near $D$ the forms $i^*_{C'}F,i^*_{C''}F$ can be represented respectively as $dA',dA''$, for suitable $1$-forms $A',A''$. It is easy to see by using Stokes' theorem that the difference $\int_DA''-\int_DA'$ must then be an integer, and must equal the algebraic sum of the degrees of all the singularities contained inside $C'\cap C''$. It is thus natural to formulate the following compatibility condition more in general:
\begin{equation*}
\begin{split}
(C^*):&\quad\forall x\text{ for a.e. circle }D\text{ with center }x, \int_DA'-\int_DA''\in\mathbb Z,\\
 & \text{where }i_{C'}^*k(C')=dA',i_{C''}^*k(C'')=dA''\text{ locally near }D.
\end{split}
\end{equation*}
It is easy to see that condition $(1)$ of Definition \ref{compatdef} is satisfied, while condition $(2)$ will probably be achievable using the techniques leading to the closure theorem \ref{closureweakbdl}. The third condition is however still to be investigated. We thus formulate the following
\begin{question}\label{star5}
 Is condition $(C^*)$ a  compatibility condition in the sense of Definition \ref{compatdef}?
\end{question}

\subsubsection{A more complex candidate for compatibility} Example \ref{tooweak} suggests considering a stronger form of condition $(C)$ than just the requirement that spherical slices have integer degree. In order to give a second candidate for a compatibility condition, we will now suggest how to extend the class $\mathcal C$ here, to include all boundaries of bounded sets writable as finite intersections of balls and of complements of balls (since we are interested in the boundaries, and just finite intersections are involved, it is not relevant whether we use closed or open balls). We will call such boundaries \textbf{convex spherical polyhedra}, in analogy  with the case when balls are replaced by half-spaces.

\subsubsection{Cell complex structure and genericity} Consider a convex spherical polyhedron (where $i\geq 1, j\geq 0$ and $B_i,B_j$ are balls included in $\Omega$):
$$
C=\partial\left[\left(\cap B_i\right)\cap\left(\cap \bar B_j^c\right)\right].
$$
A natural notion of genericity, which is also easily seen to be admissible according to Definition \ref{residual}, can defined as follows: if $B_i=B(x_i,r_i)$ then generic sets of perturbations of $C$ will be the ones formed by
$$
C'=\partial\left[\left(\cap B(x_i',r_i')\right)\cap\left(\cap \bar B(x_j',r_j')^c\right)\right],
$$
with $r_i'\in ]r_i(1-\epsilon),r_i(1+\epsilon)[\setminus N_i$, $\mathcal L^1(N_i)=0$, and with $x_i'\in B(x_i,\epsilon)\setminus A_i$, $\mathcal L^3(A_i)=0$ for all $i$ (and similarly for $j$).\\

It will also be useful to consider the natural induced cell complex structure on each cycle $C$; the dimensions of the different faces will agree with the Hausdorff dimension of the underlying sets. We will use as $2$-skeleton the sets $D_i$ such that $\mathring D_i$ are the connected components of the interior of $C\cap\partial B$ for some $B$ among the ones in the definition of $C$. The lower-dimensional skeletons can be then defined by intersection.\\
Given this cell complex structure, we define a spherical cell complex as in singular homology theory, the only differing feature being that all our $3$-cells are required to be obtained via intersections of (generic) balls. We thus also have a way of making sums and differences of our cycles. Given these data, it will be enough to define the compatibility condition on couples of convex spherical polyhedra having exactly one common face, then extend the definition by taking sums to more general cases.\\

\subsubsection{Strategy for another possible definition of compatibility}
We first define a compatibility property for two spherical polyhedra having just one common face of highest dimension, then we extend this to all polyhedra. The whole construction is done in the case of forms $F\in \mathcal F_{\mathbb Z}^p(\Omega)$ having finitely many singularities, which is a very special and easy case.\\

\textbf{Compatibility for neighboring cells.} Given $h^{C'}\in Y_{C'}, h^{C''}\in Y_{C''}$ where the $2$-skeletons $(K_{C'})^2, (K_{C''})^2$ have exactly one common face, we describe a candidate compatibility condition for $h^{C'}$ and $h^{C''}$ as follows.
\begin{itemize}
\item On the face $D\in(K_{C'})^2\cap (K_{C''})^2 $ we ask that $h^{C'}=h^{C''}$ a.e.
\item On the faces neighboring $D$ in $(K^{C'})^2$ the form $ h^{C'}$ can be expressed locally as $dA_{h^{C'}}$, and similarly $h^{C''}=_{loc}dA_{h^{C''}}$. We then ask also that $\in\mathbb Z$, where the orientations on $\partial D$ in the two integrals are coming from the orientations of $C', C''$ respectively.
\end{itemize}
\textbf{Definition of our candidate condition $(C')$.} We define the candidate property $(C')$ for a function $k$ as in \eqref{slices} in the case of a form $F\in\mathcal F_{\mathbb Z}^p(\Omega)$ which is smooth outside a finite number of singularities $a_1,\ldots,a_k\in\Omega$ and :
$$
(C'):\:\text{\emph{The above conditions hold with the choice} }h^C=k(C),\text{ \emph{for all} }C\in\mathcal C.
$$
For $F\in\mathcal F_{\mathbb Z}^p(\Omega)$ which has just isolated singularities, one can find local representations outside the singularities via potentials $A_{h^C}$ as above, and the integer $\int_{\partial D}A_{h^{C'}} - \int_{\partial D}A_{h^{C''}}$ is equal to the number of charges inside the sum cycle $C'+C''$. We formulate the future steps to be taken from here on, as an open question:
\begin{openprob}
 Is it possible to extend the definition of condition $(C')$ to the whole class $\mathcal F_{\mathbb Z}^p(\Omega)$, and if so, is the result of this extension a compatibility condition in the sense of Definition \ref{compatdef}?
\end{openprob}
\subsection{Regularity of critical points and of minimizers of the $L^p$-energy.}
We have already recalled in Proposition \ref{densitykessel} of the introduction that to any $F\in\mathcal F_{\mathbb Z}^p$ we can associate an integral rectifiable current of finite mass such that $\partial I=dF$ as distributions, and the map $dF\mapsto I$ is bounded with respect to the mass norms.\\
It is an intriguing direction of investigation to study how much of the information about $F$ is encoded in such an $I$, and to try and use instruments of the theory of currents to study the forms in $\mathcal F_{\mathbb Z}^p$. We will now consider the regularity question for minimizers of the $L^p$-energy in the class $\mathcal F_{\mathbb Z}^p$ as in \eqref{minpb}, or more generally for critical points
\begin{equation}\label{critpteq}
\begin{split}
 &F\in\mathcal F_{\mathbb Z,\phi}^p(\Omega): \int_\Omega \langle|F|^{p-2}F,X\rangle dx^3=0,\\
&\text{ for all }X\in\mathcal F_{\mathbb Z, 0}^p(\Omega)\text{ such that }dX=0\text{ in }\mathcal D'(\Omega).
\end{split}
\end{equation}
This equation is usually complemented by the stationarity requirement
\begin{equation}\label{station}
\begin{array}{l}
\text{For all continuously parameterized families }\Phi_t\in Lip(\Omega,\Omega), t\in[-\delta,\delta]\\
\text{such that }\forall t\:\Phi_t|_{\partial\Omega}=id_{\partial\Omega}\text{ and }\Phi_0=id_\Omega,\\
\text{there holds }\left.\frac{d}{dt}\int_\Omega|\Phi_t^*F\right|^pdx^3|_{t=0}=0.
\end{array}
\end{equation}
It is tempting to imitate the blow-up-and-monotonicity (or stationarity) approach in order to study the singularity points, and to prove regularity results, in the spirit of \cite{Simon}, \cite{Linblowup}. Besides the harmonic maps, another model problem for us is the regularity of minimal surfaces \cite{degiorgi}. The use of the fact that the current is boundaryless in the problem of mass-minimizing integral currents (in the case of harmonic maps, the studied quantity is a gradient, thus the property of being boundaryless is implicit there too), would correspond here to the use of the fact that $\langle \partial I,\chi_{B_r(x)}\rangle\in \mathbb Z$ is constant for a.e. $r>0$, which in our case is not necessarily true; so we must content ourselves of a weaker result, relying on the above Euler-Lagrange equations and their analogues, and trying to use the integrality of the boundary to reduce to the boundaryless case. One of the main steps in the study of our problem would be the description, and perhaps even the classification, of tangent maps, as is done in the case of minimizing hypersurfaces (where the tangent spaces are seen to be hyperplanes without much effort), and as was achieved (with more effort) in the case of harmonic maps for example in \cite{Breziscoronlieb}. The best-behaving local models of singularities in our case are up to rotation $2$-homogeneous and symmetric, but proving that they are the only possible tangent maps of critical forms $F$ is so far just a conjecture.\\
\begin{openprob}
 Suppose $F$ is a minimizer of the energy, in particular it satisfies \eqref{critpteq}, \eqref{station}, and a suitable comparison principle. Assume that the $L^p$-weak limit of a blow-up sequence $F_i(x)=r_i^2F\left(\frac{x-x_0}{r_i}\right)$, $(r_i\to 0)$ exists. Then prove that such limit equals, up to a rotation, one of the forms $\Phi_k(x)=k\frac{x}{|x|^3}, k\in\mathbb Z$. In particular the limit is unique.
\end{openprob}
 For minimizers the conjecture is true, as will be proved in some future work. Our proof however uses the results of \cite{P3}, i.e. the special properties available for minimizers. Note that the same kind of properties are used also in the result \cite{Breziscoronlieb} about harmonic maps, therefore the following question is meaningful also in that case. 
\begin{question}
 Suppose that a form $F\in\mathcal F_{\mathbb Z}^p(\Omega)$ satisfies \eqref{critpteq} and \eqref{station}. Is it possible to use just these two facts and obtain the uniqueness of tangent forms?
\end{question}
Supposing that the tangent maps are classified, we have still one more step to achieve until the local regularity becomes provable, at least if the classification is done in the sense of the above conjecture: we must use the information given by the tangent maps on degrees, in order to eliminate the possibility that singularities (i.e. points of the support of the above boundary $\partial I$ of the current of Proposition \ref{densitykessel}) are present in the regions of small rescaled energy: this would indeed be the analogue of an $\epsilon$-regularity theorem, in the setting of a classical approach to regularity (as opposed to the alternative combinatoric approach of \cite{P3}). We state the following open problem in this spirit:
\begin{question}[classical proof of $\epsilon$-regularity]\label{questionereg}
Find a proof of the $\epsilon$-regularity for minimizing forms $F$ as in the problem \eqref{minpb}, using just the weak equation \eqref{critpteq} and not the approximation result as in \cite{P3}.
\end{question}
We spend the rest of this subsection to give some hints encouraging the idea that the answer to the above question is positive. \\

The above problem can transformed into an abstract question involving only the more handy current $I$, rather than the mysterious form $F$. We will concentrate on the following property involving the boundary of a finite mass integral $1$-current $I$ (the idea will be to use this property in relation to the current $I$ of Proposition \ref{densitykessel}).
\begin{definition}\label{kirchheimprop}
Suppose that $I$ is an integer multiplicity rectifiable $1$-current of finite mass on $\Omega$, whose boundary is defined in the sense that the flux (or average of the slice done via the distance function, in the terminology of \cite{Federer})
\begin{equation}\label{eqbdryI}
\begin{array}{l}
\phi(B_r(x)):=\langle I,d\chi_{B_r(x)}\rangle=\langle I,\op{dist}(x,\cdot),r\rangle(1)\\
\text{ is well-defined and belongs to }\mathbb Z\\
\text{ for all }x,\text{ a.e. }r>0\text{ such that }B_r(x)\subset\Omega.
\end{array}
\end{equation}
We denote by $(\mathcal P)$ the property that for all $x\in\Omega$ there exist a strictly decreasing sequence $r_i^x\to 0$ such that for all $i$ $\phi(B(r_i^x,x))$ is well-defined and equal to zero.
\end{definition}
The result which connects this definition to the above program of solving the regularity question in the harder way, can be formulated as follows:
\begin{proposition}\label{propp}
Let $F\in\mathcal F_{\mathbb Z}^p(\Omega)$, $\Omega\Supset B_1(0)$, and let $I$ be the integral current as in the above Proposition \ref{densitykessel}. Suppose that $F$ minimizes the $L^p$-norm with constrained boundary trace, as in \eqref{minpb}.
There exists then a constant $\epsilon_0>0$ which is independent of $F, I, \Omega$ such that if
\begin{equation}\label{ereg}
\int_{B_1(0)}|F|^pdx^3\leq \epsilon_0,
\end{equation}
then on the smaller ball $B_{3/4}(0)$ the current $I$ has property $(\mathcal P)$.
\end{proposition}
\begin{rmk}
Since the radius is $1$ and since we do not impose constraints on the size of $\Omega$, the above energy \eqref{ereg} on $B_1(0)$ can be considered to correspond to the correct scale-invariant version of energy, which in general has the form $r^{-\alpha}\int_{B_r(x)}|F|^pdx$. In our case (see \cite{P3}) the correct choice is $\alpha=3-2p$. 
\end{rmk}
\begin{proof}[Proof of Proposition \ref{propp}:]
We will use the following result:
\begin{itemize}
 \item \textbf{Monotonicity formula (\cite{P3}, Proposition 5.2):} if $F\in\mathcal F_{\mathbb Z}^p(\Omega)$ is stationary (i.e. it satisfies \eqref{station}) then for $B_r(x)\subset\Omega$ we have
\begin{equation}\label{eq:monotonicityend}
 \frac{d}{dr}\left(r^{2p-3}\int_{B_r}|F|^pdy\right)=2p\;r^{2p-3}\int_{\partial B_r}|F|^{p-2}|\partial_\rho\lrcorner F|^2d\sigma
\end{equation}
where $\partial_\rho=\frac{\partial}{\partial\rho}$ is the radial derivative in $B_r(x)$.
\end{itemize}
By integrating the above formula \eqref{eq:monotonicityend} we obtain that the rescaled energy $E_r(F,B_r(x))$ is increasing in $r$, where we use the notation $E_a(F,B)=a^{2p-3}\int_B |F|^pdx$. We now use the small energy assumption on $B_1(0)$ in order to obtain a rescaled energy bound on balls $B_r(x)\subset B_1(0)$, with $x\in B_{3/4}(0)$: 
\begin{eqnarray*}
 \epsilon_0&\geq& E_1(F,B_1(0))\geq C E_{\bar r}(F, B_{\bar r}(x))\\
&\geq& C E_r(F,B_r(x))\quad\forall r\text{ s.t. }B_r(x)\subset B_1(0),
\end{eqnarray*}
where $\bar r=\op{dist}(x, \Omega\setminus B_1(0))$. Now suppose that $(\mathcal P)$ were false; we then have that for some $x$ there exists $r(x)>0$ such that for almost all $r\in]0,r(x)[$ there holds $\left|\int_{\partial B_r(x)}i^*F\right|\geq 1$, which implies
\begin{eqnarray*}
E_r(F,B_r(x))&\geq& r^{2p-3}\int_0^r\int_{x+r\omega\in\partial B_r(x)}|i^*F(x+r\omega)|^pr^2d\omega dr\\
&\geq& Cr^{2p-3}\int_0^rr^{2-2p}dr=C. 
\end{eqnarray*}
We have used just the inequality $|F|\geq|i^*F|$ and the Jensen inequality on each $\partial B_r(x)$. The constants so far depend just on $p$ and on the dimension, so we can find $\epsilon_0>0$ independent of $F,I,\Omega$ such that the two above inequalities are incompatible. Thus the proof is finished. 
\end{proof}
We now describe a positive result suggesting that from Proposition \ref{propp} it could be possible to obtain a positive answer to the Question \ref{questionereg} above. More precisely, we show that a stronger version of property $(\mathcal P)$ of Definition \ref{kirchheimprop} is enough to obtain regularity. This result was suggested to us by Bernd Kirchheim, but it will be evident that the proof does not extend to include assumptions that are as weak as in property $(\mathcal P)$.
\begin{lemma}\label{kirchlemma}
Let $I$ be an integral flat $1$-current of finite mass such that property $(\mathcal P)$ holds, and the function $r(x)$ of Definition \ref{kirchheimprop} satisfies a uniform lower bound $r(x)>c$. Then $\partial I=0$.
\end{lemma}
\begin{proof}The fact that 
$\langle I,d\phi\rangle=0$ for all test functions $\phi$ supported on some ball of radius $c$, follows easily for $\phi$ of the form $\phi(x)=f(|x-p|)$, because their superlevelsets are small enough balls.\\

We want now to prove that when $\phi\in C^1_c$ is a general test function, $\langle\partial I,\phi\rangle=0$. To this end we start with a function $\phi_N=\sum_{i=1}^Na_i\chi_{B_i}$ where the radii of the $B_i$ are at most $c/2$, and take a family of radial mollifiers $\rho_\epsilon,\epsilon>0$, supported in balls of radius $c/2$ (this can be achieved for $\epsilon>0$ small enough). Then since $\eta_{\epsilon,i}:=\chi_{B_i}*\rho_\epsilon$ is radial and compactly supported in a ball of radius at most $c$, we obtain
$$
0=\sum_{i=1}^Na_i\langle\partial I, \eta_{\epsilon, i}\rangle=\langle \partial I,\phi_N*\rho_\epsilon\rangle.
$$
Now we claim that $\phi_N\to\phi$ in $L^1$ implies $\phi_N*\rho_\epsilon\to\phi *\rho_\epsilon$ in $C^1$-norm. Indeed,
$$
\partial_i\int\psi(x-y)\rho_\epsilon(y)dy=\pm\int \psi(z)\partial_i\rho_\epsilon(x-z)dz,
$$
and for $\psi=\phi_N-\phi$ we can estimate the absolute value of the above integral by 
$$
||\partial_i\rho_\epsilon||_{L^\infty}||\phi_N - \phi||_{L^1},
$$
which converges to zero as $N\to\infty$.\\

Similarly we can prove $d(\phi *\rho_\epsilon)\stackrel{C^0}{\to}d\phi$: we can estimate 
$$
\left|\int\partial_i\left[\phi(x-y) - \phi(x)\right]\rho_\epsilon(y)dy\right|\leq\omega(\epsilon)||\rho_\epsilon||_{L^1}\to 0,
$$
where $\omega(t)$ is the modulus of continuity of $d\phi$. Now testing all the above convergences with $I$, we obtain the wanted 
$$
\langle I,d\phi\rangle=0.
$$
\end{proof}
 Observe that for a doubling locally finite measure, to be zero on a fine covering (as are the balls in the definition of property $(\mathcal P)$) is equivalent to being zero, by the Vitali covering theorem (see Chapter 2 of \cite{Federer}). Thus if we knew for example that $\partial I$ were a measure, we would easily conclude. On the other hand, in our case $\partial I$ being a locally finite measure is equivalent to $F$ having just a locally finite number of singular points, which is the statement of the partial regularity \cite{P3}; we thus prefer not to use it as an assumption. We formulate instead the following intriguing abstract question:
\begin{openprob}
 Suppose $I$ is an integer rectifiable current of finite mass. Does property $(\mathcal P)$ imply the fact that $\partial I=0$?
\end{openprob}
In order not to mislead the reader, we observe that currents $I$ as in Proposition \ref{densitykessel} are \emph{more than just integer rectifiable of finite mass} in (our) case $p>1$, since they automatically conserve also some of the information on the higher integrability of $F$. This is the spirit of the following easy counterexample (the idea is the same as for Example 7.1 of \cite{P1}).
\begin{lemma}
 It is possible to find a rectifiable integer $1$-current $J$ of finite mass satisfying the following properties:
\begin{itemize}
 \item The slice $\langle J, d(\cdot,x),r\rangle$ exists and gives an integer when tested with the constant $1$ for all $x$ and all but at most countably many $r$.
 \item For no $F\in\mathcal F_{\mathbb Z}^p(\Omega)$ with $p>1$ is the current $I$ given by Proposition \ref{densitykessel} equal to $J$.
\end{itemize}
\end{lemma}
\begin{proof}
 We may suppose up to rescaling that $\Omega\Subset Q$ where $Q$ is a square of sidelenght $2$ (this is more than needed for the rest of the proof, but makes the notations easier). Consider a sequence of positive numbers $a_i$ such that $\sum_ia_i=1, \sum_i a_i^{3-2p}=\infty$, then align a sequence of disjoint balls $B_i$ such that $\op{diam}B_i=2a_i$, along one of the axes of $Q$. Inside each $B_i$ two disjoint spheres $S_i^\pm$ of diameter $a_i$ can be packed; then identify the oriented segment joining their centers with an integer $1$-current, and define $I$ as the sum of all these currents. The $I$ has mass $1$, but (by applying Jensen's inequality on smaller spheres concentric to the $S_i^\pm$) it is evident that $F$ as in the statement of the Lemma must have $\int_Q|F|^pdx\geq C\sum_ia_i^{3-2p}=\infty$. This concludes the proof.
\end{proof}

\bibliographystyle{amsalpha}

\end{document}